\documentclass[a4paper,12pt,leqno]{amsart}
\usepackage{txfonts}
\usepackage[T1]{fontenc}
\usepackage[latin1]{inputenc}
\usepackage{graphicx}
\usepackage{mathrsfs}
\usepackage[english]{babel}
\usepackage{amssymb,hyperref,amsmath}
\theoremstyle{plain}
\newtheorem{theo}{Theorem}[section]
\newtheorem{prop}[theo]{Proposition}
\newtheorem{lemm}[theo]{Lemma}
\newtheorem{coro}[theo]{Corollary}
\newtheorem*{maintheo}{Main Theorem}

\theoremstyle{definition}
\newtheorem{defi}[theo]{Definition}

\DeclareMathAlphabet{\mathrmsl}{OT1}{cmr}{m}{sl}
\renewcommand{\geq}{\geqslant}

\DeclareMathOperator{\Ric}{Ric}

\newcommand{\Scal}{\operatorname{Scal}}

\newcommand{\Ker}{\operatorname{Ker}}
\newcommand{\Ad}{\operatorname{Ad}}

\newcommand{\ad}{\operatorname{ad}}
\renewcommand{\geq}{\geqslant}

\renewcommand{\exp}{\operatorname{exp}}

\newcommand{\Range}{\operatorname{Range}}
\newcommand{\Int}{\operatorname{Int}}

\newcommand{\GG}{\mathscr{G}}
\newcommand{\lig}{\mathfrak{g}}
\newcommand{\lip}{\mathfrak{p}}
\newcommand{\grlig}{\operatorname{gr}(\mathfrak{g})}
\newcommand{\grTM}{\operatorname{gr}(TM)}

\newcommand{\proofof}[1]{\end{#1}\begin{proof}}

\newcounter{mnotecount}[section]
\renewcommand{\themnotecount}{\thesection.\arabic{mnotecount}}
\newcommand{\mnote}[1]
{\protect{\stepcounter{mnotecount}}$^{\mbox{\footnotesize  $
      \bullet$\themnotecount}}$ \marginpar{\raggedright\tiny\em
    $\!\!\!\!\!\!\,\bullet$\themnotecount: #1} }

\swapnumbers

\begin{document}
\title[Parallel curves in parabolic geometries]{Parabolic geodesics as parallel curves\\ in parabolic geometries}
\author{Marc Herzlich} 
\address{Institut de math\'ematiques et de mod\'elisation de Montpellier\\ CNRS \& Universit\'e Montpellier~2} 
\email{herzlich@math.univ-montp2.fr}
\date{July 17th, 2012}

\begin{abstract} 
We give a simple characterization of the parabolic geodesics introduced by \v{C}ap, Slov\'ak and \v{Z}\'adn\'{\i}k for all parabolic geometries. 
This goes through the definition of a natural connection on the space of Weyl structures. We then show that parabolic geodesics can be characterized as 
the following data: a curve on the manifold and a Weyl structure along the curve, so that the curve is a geodesic for its companion Weyl structure
and the Weyl structure is parallel along the curve and in the direction of the tangent vector of the curve.
\end{abstract}

\maketitle

\section*{Introduction}
Parabolic geometries have attracted much interest in the recent years, providing an efficient framework for 
tackling problems in various geometries as, \emph{e.g.}, those induced by conformal or Cauchy-Riemann (CR) structures. 
There exists curves in parabolic geometries (hereafter called parabolic geodesics) that play a similar role as 
geodesics in Riemannian geometry: for instance, they provide natural local charts adapted to the geometry at hand, 
see \cite{frances} for applications.

\smallskip

In conformal geometry, conformal geodesics may have been known to Cartan 
\cite{cartan-espaces-connexion-conforme} and their study has been revived in the 80's by Ferrand 
\cite{ferrand-geodesiques-conformes}. For CR structures, canonical curves, or chains, were defined
by Cartan \cite{cartan-dim3,cartan-dim3bis} in dimension $3$, and by Chern and Moser \cite{chern-moser} in any dimension. Another
definition was then introduced by Fefferman \cite{fefferman-monge-ampere} and more general curves were studied by Koch
\cite{Koch1,Koch2}. A general definition of distinguished curves for all parabolic geometries, or \emph{parabolic geodesics}, 
was then given by \v{C}ap, Slov\'ak, and \v{Z}\'adn\'{\i}k \cite{cap-slovak-book,cap-slovak-zadnik-curves}. 

\smallskip

However, and contrarily to geodesics, parabolic geodesics are not defined directly as solutions of a natural differential equation
on the manifold itself, but rather as projections of special curves on a natural 
principal bundle (the so-called Cartan bundle) over the base manifold. To the knowledge of the author, 
natural differential equations for parabolic geodesics have been derived only in the case of conformal or CR geometry 
\cite{bailey-eastwood-conformal-circles,chern-moser,fefferman-monge-ampere,ferrand-geodesiques-conformes,Koch1,Koch2}.
The goal of this short paper is then to provide such a simple definition. To achieve this, we shall need to study in some detail a special class
of connections in parabolic geometries: the Weyl structures, introduced by \v{C}ap and Slov\'ak 
\cite{cap-slovak} and independently by Calderbank, Diemer and Sou\v{c}ek \cite{dmjc-ricci}. The set of Weyl structures forms a bundle, and we shall 
show that this bundle admits a natural connection. In a second step, we shall prove the following

\begin{maintheo} Distinguished curves (parabolic geodesics) on a manifold endowed with a parabolic geometry are exactly the curves 
that satisfy the following requirements: 
\begin{itemize}
\item[(i)] they are parallel for a Weyl structure~;
\item[(ii)] this Weyl structure is itself parallel along the curve and in the direction of the tangent vector of the curve 
for the natural connection acting on the bundle of Weyl structures.
\end{itemize}
\end{maintheo}

The paper is organised as follows. Section 1 recalls the basic facts on parabolic geometries and their Weyl structures. Sections 2 and 3 yield 
the definition
of the natural connection on the bundle of Weyl structures. Our main result is proved in section 4. The paper ends in section 5 with a few
examples of computations, where we shall re-derive from our setting known equations in the most simple cases.

{\flushleft\emph{Acknowledgements}. -- This note emerged from a series of discussions with Paul Gauduchon and the author is happy to acknowledge his
influence on this work. He also thanks Olivier Biquard and Charles Francès for their interest and comments.}

\medskip

\section{Parabolic geometries}

\smallskip

We recall here a few basic facts on parabolic geometries. Our main references in this section are  \cite{dmjc-ricci,cap-slovak,cap-slovak-book}.
In all what follows, we shall denote by $M$ a connected manifold (of dimension $n$), by $G$ a (real) semi-simple Lie group, and 
by $P$ a closed parabolic subgroup of $G$, whose Lie algebras will be denoted by $\mathfrak{g}$ and $\mathfrak{p}$. 
This implies the existence of a filtration of $\mathfrak{g}$:
$$ \mathfrak{g} = \mathfrak{g}_{(-k)} \supset \mathfrak{g}_{(-k+1)} \supset \cdots \supset \mathfrak{g}_{(0)} = \mathfrak{p} \supset 
\mathfrak{g}_{(1)} \supset \cdots \supset \mathfrak{g}_{(k)} $$ 
for some positive integer $k$. Here one has $\mathfrak{g}_{(1)} = [\mathfrak{p},\mathfrak{p}]$ (thus $\lig_{(1)}$ is the maximal nilpotent 
ideal in $\lip$ and one has $\lig_{(1)} = \lip^{\perp}$ with respect to the Killing form of $\lig$), 
$\mathfrak{g}_{(2)} =  [\mathfrak{g}_{(1)},\mathfrak{g}_{(1)}]$, and the subsequent $\mathfrak{g}_{(j)}$ for $j>0$ are the elements of the descending
central series of $\mathfrak{g}_{(1)}$; for $j<0$, the terms of the filtration are defined by $\lig_{(j)} = (\lig_{(-j-1)})^{\perp}$.
The associated graded Lie algebra $\grlig$ factors as
$$ \grlig = \lig^{-k} \oplus \lig^{-k+1}  \oplus \cdots \oplus \lig^{0} \oplus \lig^{1}  \oplus \cdots \oplus  \lig^{k} $$
where $\lig^j = \lig_{(j)}/\lig_{(j+1)}$. Thus, $[\lig^{i},\lig^{j}]\subset\lig^{i+j}$ for any $i$
and $j$ in $\{-k,\ldots,k\}$. 

\smallskip

Any choice of a Cartan involution in $\lig$ induces a Lie algebra isomorphism between $\lig$ and $\grlig$. There is no canonical 
choice of a Cartan involution, but any two such choices are related by the adjoint action of an element of $P$.
For reasons that will be clear below, the choice of such an identification will be harmless for our purposes, hence we
shall assume that such a choice has been made once and for all. It will also be useful to consider the subalgebras 
$\lig^{+} = \lig^{0} \oplus \lig^{1}  \oplus \cdots \oplus  \lig^{k}$ and $\lig^{-} = \lig^{-k} \oplus \cdots \oplus \lig^{-1}  \oplus  \lig^{0}$
in $\lig$. The smaller subalgebras $\lig^{>0}$ and $\lig^{<0}$ are defined by removing the $0$-th factor in the above formulas. 
We shall denote by $P^0$ the closed subgroup of $P$ whose adjoint action preserves the grading of $\lig$ into the
$\lig^0$-modules $\lig^i$ ($i=-k,\ldots,k$); its Lie algebra is of course $\lig^0$. In the same vein, $P^{>0}=\exp(\lig^{>0})$ is
the connected subgroup of $P$ whose Lie algebra is $\lig^{>0}$. It is then clear that $P_0 = P/P^{>0}$ is isomorphic to $P^0$; this 
is the group known as the \emph{Levi factor}. 

\begin{defi}
A parabolic geometry of type $(G,P)$ on $M$ is given by a $P$-principal bundle $\mathscr{G} \rightarrow M$ and
a section $\omega$ (a \emph{Cartan connection}) of the bundle of $1$-forms on $\mathscr{G}$ with values in $\mathfrak{g}$ such that
\begin{itemize}
\item[(a)] $\omega$ is equivariant with respect to the natural action of $P$, \emph{i.e.}
$$ (R_p)^*\omega = \Ad (p^{-1})\circ \omega \quad \textrm{ for any } p\in P\, ;$$
\item[(b)] $\omega$ reproduces the Maurer-Cartan form of $P$ on vertical vector fields~;
\item[(c)] $\omega_e : T_e\mathscr{G}\rightarrow  \mathfrak{g}$ is an isomorphism at any $e$ in $\mathscr{G}$.
\end{itemize}
\end{defi}

Under a mild cohomological condition on the Lie algebras $\lig$ and $\lip$, the geometric structure can be entirely described 
at the bottom level, \emph{i.e}. on the manifold $M$ itself, as a filtration of the tangent bundle $TM$ of the form
$$ TM = T_{(-k)}M \supset \cdots \supset T_{(-1)}M \ ,$$
induced by the identification of $TM$ with the associated bundle $\GG\times_P \lig/\lip$, together with a reduction of the structure 
group of the associated graded bundle $\grTM$ to $P_0$. A fundamental theorem due to \v{C}ap-Schichl \cite{cap-schichl} 
(see also \cite{cap-slovak-book} for a presentation in book
form with a different proof) asserts that under another, similar, 
algebraic condition on the pair $(\lig,\lip)$, any such filtered structure on $TM$ which satisfies 
a simple compatibility condition between the Lie algebraic bracket of $\lig$ and the bracket of vector fields (this condition
is called regularity) induces a parabolic geometry in the sense given above. Moreover, the latter is unique up to isomorphism
if the Cartan connection is required to be \emph{normal}, another algebraic condition on the curvature of $\omega$. All classical examples 
of parabolic geometries (\emph{e.g}. conformal geometries of any signature, CR geometry, etc.) belong to this setting. Another example is
projective geometry (\emph{i.e.} the geometry induced by a projective class of torsion-free connections on the tangent bundle), with the 
special feature that the first cohomological condition allued to above is not satisfied in this case, hence the underlying geometric structure
does not reduce to a filtration of the tangent bundle and a reduction of the structure group of the induced graded bundle; apart from this, it 
however fits into the parabolic setting rather nicely, and we shall also discuss this example later.

\smallskip

The last tool for our construction will be a subgroup $P^-$ of $G$ whose Lie algebra is $\lig^-$. We shall moreover assume that there is such
a subgroup so that $P\cap P^- = P^0$. This is a mild algebraic condition which is satisfied in all classical cases; note that it is also
satisfied when all three groups $P$, $P^0$, and $P^-$ are connected.

\medskip

\section{Weyl structures}\label{sec:weyl-structures}

\smallskip

Similarly to the whole family of Levi-Civita connections of al metrics in a conformal class, very parabolic geometry admits 
a family of distinguished connections called \emph{Weyl structures}. There are different ways to define them, but 
the following, due to \v{C}ap and Slov\'ak \cite{cap-slovak}, will be the most useful here.

\begin{defi}
A Weyl structure is a reduction $\GG^0$ of the $P$-principal bundle $\GG$ to the structure group $P^0$. 
This can be equivalently described as a $P^0$-equivariant bundle map $\sigma : \GG_0\rightarrow \GG$, where $\GG_0 = \GG / P^{>0}$.
\end{defi}

Any Weyl structure induces a principal connection on $\GG$, whose construction
is done as follows~: at any point $e=\sigma(e_0)$ of the image of $\sigma$ in $\GG$, the tangent bundle $T\GG$ splits as
$$ T_e\GG = T^V_e\GG \oplus T^H_e\GG $$
where $T^V_e\GG$ is the vertical bundle and $ T^H_e\GG = (\omega_e)^{-1}( \lig^{<0} )$. Since $\omega$ is equivariant and $P^0$ stabilizes
every graded subspace $\lig^j$ in $\lig$, this splitting is invariant by the right action of $P^0$ on $\sigma(\GG_0)$ and it can be extended as 
an equivariant connection on $\GG$ by the right action. This of course yields connections on bundles associated to $\GG$ through
a representation of $P$, such as, \emph{e.g.}, the tangent or cotangent bundles of the base manifold $M$. We shall term it the \emph{Weyl connection}
induced by $\sigma$ to make it clear that it depends on the choice of $\sigma$.

\smallskip

An equivalent definition of Weyl structures has been given in \cite{dmjc-ricci}. Weyl structures are
defined there directly as special $P$-equivariant splittings of $T_e\GG$ compatible with the natural
filtration of $\lig$. It is an easy exercise to show that any such splitting is necessarily induced by a unique section of $\GG$ over
$\GG_0$, thus yielding the equivalence between the two definitions.

\smallskip

From now on, we shall make extensive use of bundles associated with the Cartan bundle $\GG$. We shall need the following notation: for any space $X$
endowed with a $P$-action, the bundle $\GG\times_P X$ is the quotient of the product space $\GG\times X$ by the 
equivalence relation $(e,x) \sim (ep, p^{-1}x)$ for all $p\in P$.
We denote by $[e,x]$ the class of $(e,x)$ in $\GG\times_P X$, and sections of $\GG\times_P X$ are given by $P$-equivariant maps from
$\GG$ to $X$. 

\smallskip

It is well known that reductions to $P^0$ of the structure group of a $P$-principal bundle are classified by sections of the associated
bundle 
$$\mathscr{W}=\GG\times_P P/P^0.$$ 
Indeed, given a reduction $\sigma$, one can build a map \th\ from $\GG$ to $P/P^0$ which 
is $P$-equivariant as follows: one begins by fixing $\textrm{\th}(\sigma(e_0)) = [1]_{P/P^0}$
at any point $\sigma(e_0)$ in the image of $\sigma$; note that here and henceforth, an element between brackets 
will always denote its class in the quotient space indicated as an index to the right bracket. 
One then extends \th\ equivariantly by letting
$$ \textrm{\th}(\sigma(e_0)p) = p^{-1} [1]_{P/P^0} = [p^{-1}]_{P/P^0} .$$ 
In other words, $\sigma$ induces a section $\tau = \left[ \sigma(e_0),[1]_{P/P^0}\right] = \left[ e, \textrm{\th}(e)\right]$ 
of the bundle $\mathscr{W}$.

\begin{defi}
The bundle $\mathscr{W}$ is called the \emph{bundle of Weyl structures} of the parabolic geometry.
\end{defi}

The space of sections of $\mathscr{W}$ admits a natural right action of 
the space of sections of the bundle in nilpotent groups $\GG_0\times_{P^0} P^{>0} = \GG\times_P P^{>0} $, where the action 
of $p$ in $P$ or in $P^{0}$ on an element $n$ of $P^{>0}$ is by inner automorphisms: $n\longmapsto \Int_p(n) = p n p^{-1}$. 
Hence our claim here is that there is a right action of the space of sections of $\GG\times_{\Int} P^{>0}$ on 
$\mathscr{W}=\GG\times_{L} P/P^0$ (where $L$ stands for the left
action). Indeed, a section of $\GG\times_{\Int} P^{>0}$, seen as an equivariant map
$n : \GG \rightarrow P^{>0}$, acts on the right on 
$$ \textrm{\th} : \GG \rightarrow P/P^{0} , \quad \textrm{\th}(e) = [p(e)]_{P/P^0} $$ 
to yield a map 
$$ \textrm{\th}\odot n : \GG \rightarrow P/P^{0} , \quad \left(\textrm{\th}\odot n\right)(e) = [n^{-1}(e)p(e)]_{P/P^0} $$
This last section of $\mathscr{W}$ corresponds to the Weyl structure $\sigma\cdot n$ where the right factor $n$ acts by
right multiplication on the map $\sigma : \GG_0 \rightarrow  \GG$. This description of the action of $\GG\times_P P^{>0}$ on 
the bundle of Weyl structures is equivalent to that given in \cite[\S 2.4--2.5]{dmjc-ricci}.

\smallskip

The bundle $\mathscr{W}$ is associated to the principal bundle $\GG$ through a $P$-action on $P/P^0$ which does not extend to a $G$-action. 
Since $\GG$ does not have a true connection but only a Cartan connection, it does not induce a natural connection on $\mathscr{W}$. Of course,
any Weyl structure induces a connection on $\mathscr{W}$, but since there is no natural choice of a Weyl structure this does not help for our purposes.
However, there is a special class of associated bundles that are naturally endowed with a canonical connection: these are the bundles induced from a 
$G$-action.

\medskip

\section{The natural connection on the space of Weyl structures}

\smallskip

Let $X$ be a manifold endowed with an action of the group $G$. The bundle $\GG\times_P X$ is then endowed with a natural 
connection, which can be defined as follows: $\GG\times_P X$ is a quotient 
of the product space $\GG\times X$ and the kernel of the differential of the projection map 
$\tilde{\pi} :\GG\times X \rightarrow \GG\times_P X $ at $(e,x)$ is the range of the injective map
$$ \lip \rightarrow T_{(e,x)}\left( \GG\times X \right)  = T_{e}\GG\times T_{x} X , \quad \xi \mapsto (\tilde{\xi}_e , - \breve{\xi}_x ) $$
where $\tilde{\xi}_e$ and $\breve{\xi}_x$ are the natural vectors induced by the actions: 
$$ \tilde{\xi}_e = \frac{d}{dt}\left( e\, \exp(t\xi)\right)_{|t=0}  \ \ \textrm{ and } \ 
\ \breve{\xi}_x = \frac{d}{dt}\left(\exp(t\xi)\, x\right)_{|t=0}. $$
Any element $\xi$ of $\lip$ yields a vector field $\tilde{\xi} = \omega^{-1}(\xi)$ on $\GG$, which coincides with the
previously defined $\tilde{\xi}$. Thus, the kernel of the projection map is also identified to the range of the map $\xi 
\mapsto (\omega_e^{-1}(\xi) , - \breve{\xi}_x)$ from $\lip$ to $T_{e}\GG\times T_{x}X$.

\begin{defi} 
The horizontal space $T^H\left(\GG\times_P X\right)$ is the projection through $\tilde{\pi}$ of
$T^H\left(\GG\times X\right) = \Range \left( \xi \in \lig \mapsto (\omega_e^{-1}(\xi) , - \breve{\xi}_x ) \right)$.
\end{defi}

Thus, the horizontal space $T^H\left(\GG\times_P X\right)$ is naturally isomorphic to the quotient space 
$T^H\left(\GG\times X\right) / \Ker \tilde{\pi}_*$. Its dimension equals that of $\lig / \lip$ hence that of $M$,
and it projects itself isomorphically to $TM$ through the map
$$ (\omega_e^{-1}(\xi) , - \breve{\xi}_x) \mapsto  \pi_* \left( \omega_e^{-1}(\xi) \right) $$ 
tangent to the bundle projection $\pi : \GG \rightarrow M$. The connection induced by this choice of horizontal space will be called
\emph{the canonical (tractor) connection of the space} $\GG\times_P X$.

\smallskip

We can now apply this construction to the case where $X = G/P^{-}$. The tangent space at $[g]_{G/P^-}$
to $X$ is the image of the map $\xi \mapsto \breve{\xi}_{[g]}$ already defined above.
By analogy with conformal geometry (where $G/P^{-}$ is the Moebius sphere up to a covering), we shall make the following definition.

\begin{defi} 
The bundle $\mathscr{S} = \GG\times_P G/P^{-}$ is called the \emph{Moebius bundle} of the parabolic geometry.
\end{defi} 

\begin{prop}
The bundle of Weyl structures $\mathscr{W}$ is naturally embedded as an open set in the Moebius bundle $\mathscr{S}$.
\end{prop}

\begin{proof}
This is clear since one may send any element $\left[ e, [p]_{P/P^0} \right]$ of $\mathscr{W}$ to $\left[ e, [p]_{G/P^-} \right]$,
which is a well defined map from a $P$-bundle to another. The embedding is nothing but the natural embedding of 
$P/P^0$ in $G/P^-$ in each fiber, whose differential is obviously injective.
\end{proof}

\begin{coro}
The bundle of Weyl structures admits a natural tractor connection inherited from its embedding as an open set of the Moebius bundle.
\end{coro}

We shall now denote by $\mathscr{D}$ the tractor connection on Weyl structures. Choose now a (local) Weyl structure $\sigma$ (thought as
a $P^0$-equiva\-riant section $\sigma : \GG_0 \rightarrow \GG$), and a vector $X$ at a point $x$ in $M$. Then one may lift $X$ to a vector $\hat{X}_0$
in $\GG_0$ at a point $e_0$ in $\GG_0$, and the tractor derivative $ \mathscr{D}_X\sigma$ of $\sigma$ at $x$ in the direction of $X$ is given
by the following element of the cotangent bundle $T^*M = \GG \times_P \lig^{>0}$ 
$$ \mathscr{D}_X\sigma \, =\, \left[\, \sigma(e_0)\, , \textrm{proj}_{\lig^{>0}}\circ \omega_{\sigma(e_0)}(\sigma_*(\hat{X}_0))\, \right] .$$ 
With this in mind, we can now prove a helpful fact. We begin by recalling that any
choice of Weyl structure $\sigma : \GG_0 \rightarrow \GG$ induces a natural connection on $\GG$
obtained by choosing to split the tangent bundle of $\GG$ at any point in the image of $\sigma$
as the direct sum of the vertical subspace $\omega^{-1}(\lip)$ with the horizontal subspace 
$\omega^{-1}(\lig^{<0})$. Thus, it is given by the connection $1$-form
$\omega^{\sigma} = \omega^{+} = \mathrm{proj}^{\lip} \circ \omega$
where $\mathrm{proj}^{\lip}$ denotes the projection onto $\lip$ in $\lig$.

\smallskip

The Weyl connection is a $\GG_0$-connection: indeed, by
the very definition of $P^0$, the right $P^0$-action on $\GG$ preserves the decomposition of 
$T\GG = \omega^{-1}(\lig)$ into its graded components. Hence the connection preserves the
$P^0$-bundle given by the image of $\sigma$. Equivalently, the Weyl connection can be pulled
back to $\GG_0$, and the connection $1$-form associated to this connection on $\GG_0$ is
given by the projection $\omega^{0}= \mathrm{proj}^{\lig^0}\circ\omega $ onto $\lig^0$ of $\omega$.

\begin{lemm} \label{lemm-failure}
The tractor derivative of a Weyl structure $\sigma :\GG_0 \rightarrow \GG$
measures the failure of the Weyl structure to lift horizontal vectors in $\GG_0$ to horizontal vectors in $\GG$. 

More precisely, for any vector $Y$ in $TM$ and any horizontal lift $\hat{Y}_0$ in $\GG_0$ for the $\GG_0$-Weyl 
connection induced by $\sigma$, the vertical part of $\sigma_*(\hat{Y}_0)$ for the $\GG$-Weyl connection equals 
$\mathscr{D}_Y\sigma$.
\end{lemm}

\begin{proof}
Obvious corollary of the remarks above.
\end{proof}

\smallskip

Weyl structures whose tractor-derivative vanishes in every direction were called \emph{normal} by 
\v{C}ap and Slov\'ak \cite[Definition 5.1.12]{cap-slovak-book}.


\medskip

\section{Parabolic geodesics as parallel curves for the natural connection}

\smallskip

We can now proceed to the main result of this note. Following \v{C}ap, Slov\'ak, and \v{Z}\'adn\'{\i}k \cite{cap-slovak-book,cap-slovak-zadnik-curves},
we recall the definition of parabolic geodesics.

\begin{defi}
For any $P^0$-invariant subset $\mathfrak{a}$ of $\lig^{<0}$, let 
$$\mathscr{U}^{\mathfrak{a}} = \{ \ \xi \in \lig\setminus\{0\} \ | \ \exists p \in P \textrm{ such that } \Ad_{p^{-1}}(\xi) \in \mathfrak{a} \ \}.$$
If $\mathfrak{a}=\lig^{<0}$, then we shall simply denote $\mathscr{U}=\mathscr{U}^{\lig^{<0}}$.
\end{defi}

We recall that a \emph{constant} vector field on $\GG$ is a section of $T\GG$ of the form $\tilde{\xi}= \omega^{-1}(\xi)$ where $\xi$ is a
fixed element of $\lig$. 

\begin{defi}
Let $\mathfrak{a}$ be a $P^0$-invariant subset of $\lig^{<0}$. A distinguished curve in a parabolic geometry or \emph{parabolic geodesic}
of type $\mathfrak{a}$ is the projection on $M$ of an integral curve in $\GG$ of a constant vector field $\tilde{\xi} = \omega^{-1}(\tilde{\xi})$,
where $\xi$ is an element of $\mathscr{U}^{\mathfrak{a}}$ (hence it is not zero).
\end{defi}

The most natural choice is $\mathfrak{a}=\lig^{<0}$. For simplicity's sake, we shall only consider this case here and omit the type in the notation.

\smallskip

This definition provides a very general setting encompassing all examples of natural curves used in the most studied examples
of parabolic geometries such as, \emph{e.g.}, conformal geodesics of Ferrand \cite{ferrand-geodesiques-conformes} in conformal geometries 
(of any signature), or chains \cite{cartan-dim3, cartan-dim3bis,fefferman-monge-ampere} or horizontal curves of Koch \cite{Koch1,Koch2} 
in CR geometry. For more details, we refer to \v{C}ap and Slov\'ak's book \cite{cap-slovak-book}.
We can now prove our main result.

\begin{theo} \label{theo-main}
Parabolic geodesics of general type are exactly the curves satisfying the following requirements: 
\begin{itemize}
\item[(i)] they are parallel for a Weyl structure~;
\item[(ii)] the Weyl structure itself is parallel, among the curve, in the direction of the tangent vector of the curve for the natural 
tractor connection of the bundle of Weyl structures.
\end{itemize}
\end{theo}

\begin{proof}
Let $\gamma$ be a parabolic geodesic in the sense of \v{C}ap and Slov\'ak, and denote by $\hat{\gamma}$ a lift on $\GG$ that is an integral curve
of a constant vector field $\tilde{\xi}= \omega^{-1}(\xi)$, for a fixed $\xi$ in $\lig$. This immediately defines a Weyl structure along $\gamma$: 
indeed, $\hat{\gamma}$ projects itself onto a curve $\hat{\gamma}_0$ in $\GG_0$, and one may define a section 
$\sigma : \GG_0{}_{|\gamma} \rightarrow \GG_{|\gamma}$ as 
follows: for any $e_0$ in $\GG_0{}_{|\gamma(s)}$, there is $p^0$ in $P^0$ such that $e_0 = \hat{\gamma}_0(s) p^0$, and one sets 
$\sigma (e_0) = \hat{\gamma}(s)p^0$. As an example, we have $\sigma(\hat{\gamma}_0(s)) = \hat{\gamma}(s)$. 

\smallskip

One now needs to slightly change the lift: as $\xi$ is not zero and in $\mathscr{U}$, one may choose $p$ in $P$ such that $\Ad_{p^{-1}}(\xi)$ belongs to
$\lig^{<0}\setminus\{0\}$. Thus $\hat{\gamma}p$ is another lift of $\gamma$ which is an integral curve of $\Ad_{p^{-1}}(\xi)$. This shows that 
one may always suppose that $\xi$ belongs to $\lig^{<0}$. We shall now make this choice and denote by $\hat{\gamma}$ a lift of $\gamma$ satisfying this
extra property, and use the Weyl structure induced by this choice of the lift.
 
\smallskip

At $\hat{\gamma}(s)$, the horizontal space in $\GG$ for the so defined Weyl structure is $\omega_{\hat{\gamma}(s)}^{-1}(\lig^{<0})$. Recall now
that any vector on $M$ is an element of $\GG\times_P \lig/\lip$: more precisely, any $[e,\bar{\eta}]$ in $\GG\times_P \lig/\lip$ corresponds
to a vector
$$ X = \pi_* \circ \omega_e^{-1}( \eta ), $$
where $\eta$ is any lift of $\bar{\eta}$ in $\lig$. Thus, the vector field $\gamma'(s)$ may be thought as $\left[ \hat{\gamma}(s), [\xi]_{\lig/\lip}
\right]$
and its covariant derivative in its own direction with respect to the previously defined Weyl structure is given by the usual computation
$$\left[ \hat{\gamma}(s), \ad (\theta(\hat{\gamma}'(s))) ([\xi]_{\lig/\lip} \right] , $$
where $\theta$ is the connection $1$-form of the Weyl structure. But since $\omega(\hat{\gamma}')$ belongs to the space $\lig^{<0}$, the curve 
$\hat{\gamma}$ is an horizontal curve in $\GG$ and this implies that $\theta(\hat{\gamma}'(s)) = 0$. Thus the curve $\gamma$ on $M$ is a geodesic
for this choice of Weyl structure.

\smallskip

We now prove that the Weyl structure itself, seen as a section of $\mathscr{W}$, is parallel in the direction of $\gamma'$ with respect to the
tractor connection. Along $\gamma$, the section $\tau$ associated to the Weyl structure is given by
$$ \tau_{|\gamma} (\gamma(s)) = \left[ \hat{\gamma}(s) \, , \, [1]_{P/P^0} \, \right] ,$$
or, seen as a section of the Moebius bundle $\GG\times_P G/P^{-}$, 
$$ \tau_{|\gamma} (\gamma(s)) = \left[ \hat{\gamma}(s) \, , \, [1]_{G/P^{-}} \, \right] .$$
Lifting this curve to  $\GG\times G/P^{-}$, we can compute its tangent vector:
$$ (\hat{\gamma}'(s) \, , \, 0 ) = ( \omega_e^{-1}(\xi), 0).$$
Since $\xi \in \lig^{<0}$, 
$$ \frac{d}{dt} \left[ \exp(t \xi) 1 \right]_{G/P^-} = 0 $$
and it follows that $\tau$ is a horizontal section along $\gamma$.

\smallskip

Conversely, let $\gamma$ be a curve on $M$ such that there exists a Weyl structure for which the curve $\gamma$ is a geodesic and which 
is tractor-parallel along $\gamma$ in the direction of $\gamma'$.
Let $\hat{\gamma}_0(s)$ any local horizontal lift of $\gamma(s)$ to $\GG_0$. From the very definition of the Weyl
connection on $\GG_0$ (see above), the fact that $\hat{\gamma}_0'(s)$ is horizontal in $\GG_0$ is equivalent to the fact 
that $\omega^0\left(\hat{\gamma}_0'(s)\right)=0$, and since $\gamma'(s)$ is represented in 
$TM = \GG_0\times_{P^0}\lig^{<0}= \GG_0\times_{P^0}\lig/\lip$ by 
$$ \gamma'(s)\, = \, [e_0,\omega^{<0}\left(\hat{\gamma}_0'(s)\right) + \lip],$$ 
one deduces that $\omega^{<0}\left(\hat{\gamma}_0'(s)\right)$ is constant.

\smallskip

We now define $\hat{\gamma}(s) = \sigma \circ \hat{\gamma}_0(s)$. This is a lift of $\gamma(s)$ on $\GG$, and
Lemma \ref{lemm-failure} shows that $\hat{\gamma}'(s)$ is horizontal for the Weyl connection. Thus, 
$\omega^+\left((\hat{\gamma}'(s)\right)$ vanishes, and 
$$  \omega\left((\hat{\gamma}'(s)\right) = \omega^{<0}\left((\hat{\gamma}'(s)\right) $$
is a constant element of $\lig$. Consequently, $\gamma(s)$ is a parabolic geodesic.
\end{proof}

\medskip


\section{Examples}

\smallskip

We shall consider here three classical examples which historically served as compelling motivations for the study of 
general parabolic geometries: (pseudo-)conformal geometry, projective clas\-ses of torsion free connections, 
and Cauchy-Riemann structures. Of course the results are already known in all these cases, but this is precisely the 
reason why they may be considered as interesting settings to test the result of this note. Moreover, one may choose the
group $G$ to be a linear group in those three cases, and this simplifies a lot the task of computing the Cartan 
connection as one may use the natural representation to shift from the principal bundle viewpoint into a vector bundle
approach which is easier to cope with.

\smallskip

\subsection{Conformal geodesics} These are the oldest known examples of distiguished curves in a parabolic geometry. 
They appear in the 1982 paper by Ferrand \cite{ferrand-geodesiques-conformes}, 
but Bailey and Eastwood \cite{bailey-eastwood-conformal-circles} mention that they might have been known earlier to Penrose
\cite{penrose-rindler-1,penrose-rindler-2}.
Some authors moreover claim that they are not-so-hidden in Cartan's works on conformal geometry. 
We shall here show that our approach gives the usual equations for conformal geodesics as given by Bailey and Eastwod 
\cite{bailey-eastwood-conformal-circles}, Ferrand \cite{ferrand-invariant-geodesiques-conformes}, or Gauduchon \cite{pg-cartan}. 

\smallskip

If $(M,[g])$ is a conformal manifold of signature $(p,q)$ ($n=p+q\geq 3$), one may take $G=O(p+1,q+1)$ and $P$ is the isotropy subgroup 
of an isotropic ray (half-line) in $\mathbb{R}^{p+q+2}$. Thus $\lig = \lig^{-1} \oplus  \lig^0 \oplus  \lig^{1}$ where 
$\lig^{-1}\simeq \mathbb{R}^n$, $\lig^0\simeq \mathfrak{co}(p,q)$, $\lig^1\simeq (\mathbb{R}^n)^*$ and the subgroup $P^0$ of $P$
preserving the grading is $P^0=CO(p,q)=\mathbb{R}_{>0}\times O(p,q)$.

\smallskip

The Moebius bundle is more easily identified with the help of the so-called tractor or Cartan bundle \cite{beg,cap-slovak,pg-cartan}, and
the existence of such a bundle motivated our choice of groups, see Graham and Willse \cite{graham-willse-subtleties} 
for subtleties in this issue. The tractor bundle is the vector bundle $\GG\times_P \mathbb{R}^{p+q+2}$.
As it is induced by the (faithful) standard representation of $G$, the tractor connection may rather be defined on
this vector bundle rather than on $\GG$. For this purpose, we notice that there is a natural filtration
$$ 0 \subset  \Lambda \subset \Lambda^{\perp} \subset \mathbb{T} ,$$
where $\Lambda$ is an oriented isotropic line bundle and $\Lambda^{\perp}$ is its orthogonal subspace for the
natural pseudo-Riemannian metric of $\mathbb{T}$. Any choice of 
metric $g$ in $[g]$ induces a Weyl structure, which turns this filtration into a graduation
$$ \mathbb{T} \simeq_g \mathbf{1} \oplus T^*M \oplus \mathbf{1} ,$$
where each $\mathbf{1}$ denotes a trivial line bundle, the first one being nothing but the already introduced $\Lambda$. 
Under this isomorphism, the tractor connection on a section
$u = (\lambda, \alpha,\mu)$ of $\mathbb{T}$ is explicitly computed as
$$ \mathscr{D}_Xu \ = \ \begin{pmatrix}
d\lambda(X) + S^g(\alpha^{\sharp},X)  \\
D^g_X\alpha - \lambda\, g(X,\cdot) + \mu\, S^g(X,\cdot) \\
d\mu(X) - \alpha(X) 
\end{pmatrix}\ ,$$
where $S^g = \tfrac{1}{n-2}\left(\Ric^g - \tfrac{\Scal^g}{2(n-1)} g\right)$ is the Schouten tensor (sometimes called the Rho-tensor) 
of the metric $g$ and $X$ is any tangent vector. The bundle $\mathbb{T}$ is endowed with a pseudo-Riemannian metric of 
signature $(p+1,q+1)$ given by
$\mathbf{H}(u,u) = g^{-1}(\alpha,\alpha) - 2 \lambda\,\mu$
and the $P$-bundle $\GG$ is the the subbundle of the orthonormal frame bundle of $\mathbb{T}$ whose elements are isometries
from $\mathbb{R}^{p+q+2}$ to $\mathbb{T}$ sending a fixed isotropic ray $\mathbb{R}_{>0}e_0$ in $\mathbb{R}^{p+q+2}$
to the ray $\Lambda_{>0}$.

\smallskip

One may now take $P^-$ to be the subgroup of $G=O(p+1,q+1)$ fixing an isotropic ray $\mathbb{R}_{>0}e_{\infty}$, where
$e_{\infty}$ is any vector in $\mathbb{R}^{p+q+2}$ such that $\langle e_0,e_{\infty}\rangle = 1$. 
One has then $P\cap P^- = P^0$ as requested and one concludes that the Moebius bundle $\mathscr{S} = \GG\times_P G/P^{-}$ is the bundle of 
isotropic rays in $\mathbb{T}$.

\smallskip

When the signature is Riemannian ($q=0$), each fiber of the Moebius bundle consists in two copies of the Moebius sphere
$\mathbb{S}^{n}$, seen as the set of isotropic lines in $\mathbb{R}^{n+2}$ with Lorentzian signature $(+\cdots+-)$. 
The bundle of Weyl structures $\mathscr{W}$ 
is then identified to $\mathscr{S}^+\setminus \{\Lambda_{>0}\}$, where $\mathscr{S}^+$ is the connected component of $\mathscr{S}$
containing $\Lambda_{>0}$. 
For arbitrary $p$ and $q$, each fiber of $\mathscr{W}$ is the 
image of the embedding of $T^*M$ into $\mathscr{S}$ defined by $\beta \longmapsto \mathbb{R}\,(\tfrac12g^{-1}(\beta,\beta), \beta, 1)$,
which reproduces the usual conformal embbeding of the flat pseudo-Riemannian space $\mathbb{R}^{p+q}$ of signature $(p,q)$ into the set
of isotropic rays of the space $\mathbb{R}^{p+q+2}$ with signature $(p+1,q+1)$. This embedding will also be used below to compute explicitly 
the tractor derivative of a Weyl structure, as it indentifies the cotangent bundle $T^*M$ to the vertical tangent space 
$\GG\times_P (\lig/\lip^-)=\GG\times_P \lig^{>0}$ of the Moebius bundle.

\smallskip

Weyl structures may be retrieved in the following way: a choice of metric $g$ fixes a $P^0$-subbundle (hence a background Weyl
structure) given by the frames adapted to the splitting $\mathbb{T} \simeq_g \left(\mathbf{1}\oplus \mathbf{1}\right) \stackrel{\perp}{\oplus} T^*M$ 
where the trivial line factors are the same as the ones given above. 
Any other Weyl structure is induced in this trivialization by the choice of a field of $1$-forms $\alpha$. 
This yields another splitting $\mathbb{T} \simeq \left(\mathbf{1}\oplus \mathbb{R}w\right) \stackrel{\perp}{\oplus} T^*M$
where the first $\mathbf{1}$ is again the trivial line bundle $\Lambda$ but the second trivial line is now generated by 
$$ w = (\tfrac12g^{-1}(\alpha,\alpha), -\alpha, 1) $$ 
where $\alpha$ belongs to $T^*M$ 
(the sign in front of the second factor being chosen here to simplify
further computations). The bundle of orthonormal frames adapted to this decomposition is again
a $P^0$-subbundle of $\GG$ which is the image in $\GG$ of a new Weyl structure $\sigma : \GG_0 \rightarrow \GG$. 

\smallskip

To make the picture complete, we note 
that a field of $1$-forms may be seen as a section of the bundle in nilpotent groups $\GG\times_P P^{>0}$\,: the 
filtration of the Lie algebra $\lig$ has depth $1$ in the conformal case, which means that one has 
$ \lig = \lig^{-1} \oplus \lig^0 \oplus \lig^1$ and $\lig^1$ is the abelian algebra
$\mathbb{R}^{p+q}$. Hence,  $\GG\times_P P^{>0}$ is nothing but the bundle in abelian groups $T^*M = \GG\times_P \lig^{1}$
which we already know as a vector bundle. As a result, we recover the simply transitive action of the bundle of nilpotent 
groups on the bundle of Weyl structures discussed in section \ref{sec:weyl-structures}. In the conformal case, this action turns 
$\mathscr{W}$ into an bundle in affine spaces whose associated vector bundle is $T^*M$.

\smallskip

It now remains to compute explicitly the Weyl connection on the tangent bundle, but this is well-known. Starting with a 
metric $g$ in the conformal class, the Weyl connection induced by the Weyl structure translated from $g$ by a $1$-form $\alpha$ is
$$D^{\alpha}_X Y = D^g_X Y + \alpha(X)Y + \alpha(Y)X - g(X,Y)\,\alpha^{\sharp} $$ 
(if $\alpha$ had been chosen above rather than $-\alpha$ in the formula yielding the embedding of $T^*M$ into $\mathscr{S}$, then
all occurences of $\alpha$ here should be changed to $-\alpha$).
Thus one may fix a backround Riemannian metric $g_0$ (and its Levi-Civita connection $D^0$), and 
using either the fact that the projection into the $\lig^{>0}$-part of the Cartan connection
of conformal geomery is the Schouten tensor or the formulas given above for the tractor connection on $\mathbb{T}$, 
one easily deduces from the definition of the tractor connection
given above that the equations for a conformal geodesic are
$$ \begin{cases}
0 = D^{\alpha}_{\gamma'}\gamma' = D^0_{\gamma'}\gamma' + 2 \alpha(\gamma')\gamma' - g_0(\gamma',\gamma')\,\alpha^{\sharp}\, , \\
0 = S^{D^{\alpha}}(\gamma',\cdot) = S^{g_0}(\gamma',\cdot) - D^0_{\gamma'}\alpha - \frac12 g^{-1}(\alpha, \alpha)(\gamma')^{\flat} + \alpha(\gamma')\alpha \, .
\end{cases}$$
It is easily checked that either $g_0(\gamma',\gamma')=0$ at time $s=0$ and it remains so along the solution curve, or $g_0(\gamma',\gamma')$ 
never vanishes. In the first case, the first equation is turned into $D^0_{\gamma'}\gamma' + 2 \alpha(\gamma')\gamma' = 0$,
which is the (unparametrised) geodesic equation for any Weyl structure, and one recovers the set of null geodesic curves 
(which forms a conformally invariant set of curves).
In the second case, one may always solve the first equation for $\alpha$ and get
$$\alpha = \frac{1}{g_0(\gamma',\gamma')}\,\left( \left(D^0_{\gamma'}\gamma'\right)^{\flat} 
- 2 \frac{g_0(D^0_{\gamma'}\gamma',\gamma')}{g_0(\gamma',\gamma')}\,(\gamma')^{\flat}\right) .$$ 
Thus, the parabolic geodesics are exactly the solutions of the following well-known third-order equation 
\cite{bailey-eastwood-conformal-circles,ferrand-geodesiques-conformes,pg-cartan}
$$ \gamma''' + \frac32\, \frac{|\gamma''|^2_{g_0}}{|\gamma'|^2_{g_0}}\gamma' - 3\, \frac{g_0(\gamma'',\gamma')}{|\gamma'|^2_{g_0}}\gamma'' 
- |\gamma'|^2_{g_0}S^{g_0}(\gamma',\cdot)^{\sharp} + S^{g_0}(\gamma',\gamma')\gamma'\ = \ 0\, ,  $$
where we have of course denoted $D^0_{\gamma'}\gamma'=\gamma''$ and $D^0_{\gamma'}\left(D^0_{\gamma'}\gamma'\right)=\gamma'''$.

\smallskip

\subsection{Great circles in projective geometry}
We now consider the geometric structure given by a projective class of linear connections, \emph{i.e.}
a class of torsion-free connections on $TM$ sharing the same family of unparametrized geodesics. Of
course, this example is in a sense vacuous for our purposes as we expect that the family of
(unparametrized) parabolic geodesics will be the same as the common family of geodesics for all connections 
in the projective class. It is however interesting to see how they explicitly appear in the computations.

\smallskip

Following for instance \cite{beg} or \cite[chapter 4]{cap-slovak-book}, we may now take the group $G=SL(n+1,\mathbb{R})$ and the isotropy subgroup $P$ of a 
fixed ray ($\mathbb{R}_{>0}e_0$, say) in $\mathbb{R}^{n+1}$, where $n$ is the dimension of the underlying manifold. The Lie algebra $\lig$ splits as
$\lig = \lig^{-1} \oplus  \lig^0 \oplus  \lig^{1}$ where one has $\lig^{-1}\simeq \mathbb{R}^n$, $\lig^0\simeq \mathfrak{gl}(n,\mathbb{R})$ and
$\lig^1\simeq (\mathbb{R}^n)^*$. The group $P^0$ preserving the grading is isomorphic to $GL^+(n,\mathbb{R})$. Fixing a basis 
$(e_0,\ldots,e_n)$ of $\mathbb{R}^{n+1}$, one takes $P^-$ to be the subgroup of $G$ preserving the hyperplane generated by $(e_1,\ldots,e_n)$
together with the orientation induced by the $n$-form induced from the determinant of $\mathbb{R}^{n+1}$ by the choice of $e_0$.

\smallskip

There is again a tractor bundle $\mathbb{T}= \GG\times_P \mathbb{R}^{n+1}$, which contains a natural oriented line bundle $\Lambda  \subset \mathbb{T}$,
see \cite{beg}. Choosing a Weyl structure amounts to choosing an hyperplane in $\mathbb{T}$ transverse to the line $\Lambda$, and this induces 
a splitting $\mathbb{T}\simeq \Lambda \oplus (TM\otimes\Lambda)$ reproducing the reduction
of the structure group from $P$ to $P^0$. Given such a splitting (\emph{i.e.} given a backgroud Weyl structure), 
choosing another Weyl structure then amounts to choosing a $1$-form $u\in(\mathbb{R}^{n})^*$ since the induced hyperplane is the
kernel of the map: 
$$(\lambda, \xi) \in \Lambda\oplus(TM\otimes\Lambda) \ \longmapsto \ \ell(\lambda) - \ell\otimes u(\xi).$$ 
Equivalently, this defines in $\mathbb{T}^* \simeq (T^*M\otimes\Lambda^*) \oplus \Lambda^*$ the ray whose elements are 
$(-u\otimes\ell,\ell)$, where $\ell$ runs through all positive elements in $\Lambda^*$.

\smallskip

On the background manifold, the torsion-free connections on $TM$
are the same as the set of connections induced by Weyl structures: two Weyl structures differ one from another by a $1$-form, and the relation
beween the Weyl connection  $\nabla^u$ induced by the choice of a $1$-form $u$ and a background Weyl connection $\nabla$ is given by
$$\nabla^u_XY = \nabla_XY + u(X)Y + u(Y) X , \quad X \in TM,\ Y \in \Gamma(TM). $$
The tractor connection has been computed for instance in \cite{beg}. To state the result, one defines a background Weyl structure (\emph {i.e.} a 
connection $\nabla$ in the projective class) and one denotes by $\textsf{P}$ its modified Ricci tensor (or Schouten tensor) \cite[p. 1209]{beg}. 
As explained there, it is
always possible to choose $\nabla$ such that $\textsf{P}$ is symmetric, and we shall indeed do so. It is then a simple task 
to check that a curve $\gamma$ is a parabolic geodesic if it satisfies the equations
$$ \begin{cases} \nabla_{\gamma'}\gamma' + 2 u(\gamma')\gamma' =0 , \\
\nabla_{\gamma'}u - \textsf{P}(\gamma',\cdot) - u(\gamma')u =0 .
\end{cases}$$
The first equation shows that any solution is an (\emph{a priori} non affinely param\-etrized) geodesic for $\nabla$, and it is again easy to check 
that one may always choose the parametrization such that the second equation is satisfied. This substantiate our previous claim that the set of parabolic
geodesics here is exactly the same as the set of \emph{great circles}, or common geodesics to all connections in the projective class. 

\smallskip

\subsection{Parabolic geodesics in strictly pseudo-convex CR geometry}
The situation in the complex case of positive definite signature is very similar to the conformal case: if $n=2m+1$ is the dimension
of the underlying CR manifold, one may take the semi-simple group $G=SU(m+1,1)$
(formally, taking this group amounts to choosing a slightly refined description of CR geometry as a parabolic geometry,
since one also assumes existence of a $(m+2)$-root of the CR-canonical bundle) and $P$ and
$P^-$ are the isotropy groups of the complex lines $\mathbb{C}e_0$  and $\mathbb{C}e_{\infty}$ in $\mathbb{C}^{m+2}$, where we fixed
a basis $(e_0, e_1,\ldots, e_m,e_{\infty})$ and the hermitian form given by $\langle e_0,e_0\rangle = 0$, $\langle e_{\infty},e_{\infty}\rangle = 0$,
$\langle e_0,e_{\infty}\rangle =\langle e_{\infty},e_0\rangle = 1$, and  $\langle e_i,e_j\rangle = \delta_{ij}$ (Kronecker symbol)
whenever $(i,j)$ belongs to $\{0,\ldots,m,\infty\}^2 \setminus \{ (0,\infty), (\infty,0)\}$~; as a result, $P\cap P^- = P^0$.

\smallskip

As in the real case, it is easier to work with the tractor vector bundle $\mathbb{T}$, see \cite{gover-graham,mh-mpcps}.
If $(M,H,J)$ is a strictly pseudo-convex CR manifold of dimension $n=2m+1$, the tractor bundle is a rank $(m+2)$ complex vector bundle
endowed with a hermitian form of signature $(m+1,1)$. There is again an exact sequence
$$ 0 \subset  \Lambda \subset \Lambda^{\perp} \subset \mathbb{T} ,$$ 
where $\Lambda$ is an isotropic complex line. Choosing a contact form on $M$ whose kernel is the contact distribution $H$ yields a splitting
$$ \mathbb{T} \simeq_g \mathbf{1} \oplus (H^*M)^{1,0} \oplus \mathbf{1} ,$$
where the symbol $ \mathbf{1}$ now denotes a trivial complex line bundle and $(H^*M)^{1,0}$ is the space of $1$-forms of type $(1,0)$ on the 
contact hyperplane $H$. 

\smallskip

The description of Weyl structures is formally very similar to that given in the conformal case: the Moebius bundle $\mathscr{S}$ is the bundle of 
isotropic complex lines in $\mathbb{T}$, the bundle of Weyl structures is $\mathscr{S}\setminus\{\Lambda\}$, and much of the considerations
given above also apply here \emph{mutatis mutandis}. For instance, once a reference contact form $\theta$ has been fixed, one may describe 
$\mathscr{W}$ as the image of the embedding $\alpha =  \alpha_{|H} + \alpha_0\theta \mapsto (\tfrac12|\alpha_{|H}^{1,0}|^2 + i\alpha_0, \alpha_{|H}^{1,0},1)$
of $T^*M$ in $\mathscr{S}$. This reproduces the standard embedding of the Heisenberg algebra in $G/P^-$ also gives an identification between
the cotangent bundle and the vertical bundle of the Moebius bundle.

\smallskip

Explicit equations for a curve $\gamma$ to be a parabolic geodesic are much longer to obtain in this case than in the previous ones, so we shall skip 
leave the detailed computations to the reader. 
Let us just indicate that explicit forms of the tractor connection can be found in \cite[formula (3.3)]{gover-graham} or in \cite[formula (5.11)]{mh-mpcps}
(both with respect to the Tanaka-Webster connection), thus leading with some effort to the equations asserting that the Weyl structure 
associated to a choice of $1$-form $\alpha$ is tractor parallel along the curve $\gamma$. To get the full system of equations for a parabolic geodesic, one
must complement by the equation ensuring that $\gamma$ is the geodesic for the Weyl structure defined by the $1$-form $\alpha$, and this can be
calculated by using for instance \cite[\S 4]{dmjc-ricci} which easily leads to explicit formulas for Weyl covariant derivative on the tangent bundle.

\smallskip

\bibliographystyle{smfplain}

\end{document}